\documentclass[a4paper,12pt]{amsart}
\usepackage[utf8]{inputenc}
\usepackage{hyperref}
\usepackage{amsmath, amssymb}
\usepackage[alphabetic,nobysame, initials]{amsrefs}
\usepackage{graphicx,color}

\newtheorem{thm}{Theorem}[section]
\newtheorem{lem}[thm]{Lemma}
\newtheorem{cor}[thm]{Corollary}
\newtheorem{conj}[thm]{Conjecture}
\newtheorem{fact}[thm]{Fact}
\theoremstyle{definition}
\newtheorem{defn}[thm]{Definition}
\newtheorem{obs}[thm]{Observation}

\newcommand{\noshow}[1]{}
\newcommand{\Ze}{\mathbb Z}
\renewcommand{\Re}{\mathbb R}
\renewcommand{\epsilon}{\varepsilon}
\newcommand{\Ren}{\Re^n}

\newcommand{\Se}{\mathbb S}
\newcommand{\Sen}{\Se^n}
\newcommand{\Senm}{\Se^{n-1}}
\newcommand{\B}{\mathbf B}

\newcommand{\Pe}{\mathbb P}
\newcommand{\FF}{\mathcal F}
\newcommand{\GG}{\mathcal G}
\newcommand{\HH}{\mathcal H}

\newcommand{\di}{\mathrm d}

\newcommand{\st}{\; : \; }
\renewcommand{\phi}{\varphi}
\DeclareMathOperator{\inter}{int}
\DeclareMathOperator{\cl}{cl}
\DeclareMathOperator{\bd}{bd}

\DeclareMathOperator{\vol}{vol}
\DeclareMathOperator{\conv}{conv}
\DeclareMathOperator{\diam}{diam}

\newcommand{\card}[1]{\left|#1\right|}
\newcommand\facto[2][]{
  \ifx&#1& %If there is no 2nd arg.
  \refstepcounter{equation}
  \fi
  \begin{minipage}{0.09\textwidth}
  \ifx&#1& %If there is no 2nd arg.
  (\theequation)
  \else %If there is 2nd arg.
  {#1}
  \fi
  \end{minipage}
  \hfill
  \begin{minipage}{0.81\textwidth}
  \emph{
  \begin{sloppypar}
  #2
  \end{sloppypar}
  }
  \end{minipage}
  \smallskip
}

\title{Flavors of Translative Coverings}
\author[M. Nasz\'odi]{M\'arton Nasz\'odi}
\address{
M\'arton Nasz\'odi,
Dept. of Geometry,
Lorand E\"otv\"os University,
P\'azm\'any P\'eter S\'et\'any 1/C
Budapest, Hungary 1117
}
\email{marton.naszodi@math.elte.hu}
\thanks{The author acknowledges the support of the J\'anos Bolyai 
Research Scholarship of the Hungarian Academy of Sciences, and
the Hung. Nat. Sci. Found. (OTKA) grant PD104744.}

\begin{document}
\begin{abstract}
We survey results on the problem of covering the space $\Ren$, or a convex body 
in it, by translates of a convex body. Our main goal is to present a diverse 
set of methods. A theorem of Rogers is a central result, according to which, 
for any convex body $K$, the space $\Ren$ can be 
covered by translates of $K$ with density around $n\ln n$. We outline four 
approaches to proving this result. Then, we discuss 
the illumination conjecture, decomposability of multiple coverings, Sudakov's 
inequality and some problems concerning coverings by sequences of sets.
\end{abstract}

\maketitle

%\nopunct

\section{Introduction}

The problem of covering a set by few translates of another appears naturally in 
several contexts. In computational applications it may be used for divide and 
conquer algorithms, in analysis, it yields $\varepsilon$--nets, in functional 
analysis, it is used to quantify how compact an operator between 
Banach spaces is. In geometry, it is simply an interesting question on its own.

Our primary focus is to describe a representative family of methods, rather 
than giving a complete account of the state of the art.
In particular, we highlight some combinatorial ideas, and sketch some
instructive probabilistic computations.

We minimize overlap with the fundamental works of L. Fejes T\'oth 
\cite{LFT53} and Rogers \cite{RoBook64}.
B\"or\"oczky's book \cite{Bo04} is the most recent source on finite coverings.
Some of the topics covered here are discussed in more detail in the 
books \cites{PaAg95,Be10,Ma02,BMP05}. Many of the topics omitted, or only 
touched upon here (most notably, planar and three--dimensional results, lattice 
coverings and density) are discussed in the surveys 
\cites{FTG04,FTG99,FTG97,FTG93,FTG83} by G. Fejes T\'oth.

In Section~\ref{sec:wholespace}, we state 
Rogers' result, and a few of its relatives, on the existence of an economical 
covering of the whole space by translates of an arbitrary convex body. 
In Section~\ref{sec:rogersproofs}, we outline three probabilistic proofs of 
these results. In Section~\ref{sec:fractional}, we describe a 
fourth approach, which is based on an algorithmic (non--probabilistic) result 
from combinatorics. Then, in Section~\ref{sec:illum}, we discuss the problem of 
illumination. There, we sketch the proof of a result of Schramm, which is 
currently the best general upper bound for Borsuk's problem. In 
Section~\ref{sec:decomposability}, we state some of the most recent results on 
the problem of decomposability of multiple coverings. 
Section~\ref{sec:asymptotic} provides a window to how the asymptotic theory of 
convex bodies views translative coverings. Finally, in 
Section~\ref{sec:sequences}, we consider coverings by sequences of convex 
bodies.

\subsection*{We use the following notations, and terminology}

For two Borel measurable sets $K$ and $L$ in $\Ren$, let $N(K,L)$ denote the 
\emph{translative covering number} of $K$ by $L$, that is, the minimum number 
of translates of $L$ that cover $K$.

The Euclidean ball of radius one centered at the origin is 
$\B_2^n=\{x\in\Ren\st |x|^2=<x,x> \leq 1\}$, where $<.,.>$ denotes the 
standard scalar product on $\Ren$.
We denote the Haar probability measure on the sphere $\Senm=\{x\in\Ren\st 
|x|=1\}$ by $\sigma$.

A \emph{symmetric} convex body is a \emph{convex body} (that 
is, a compact convex set with non--empty interior) that is centrally symmetric 
about some point.
A hyperplane $H$ \emph{supports} a convex set $K$, if $H$ intersects the 
boundary of $K$, and $K$ is contained in one of the closed half--spaces bounded 
by $H$.
The \emph{support function} $h_K$ of a convex set $K$ is defined as 
$h_K(x)=\sup\{<x,k> \st k\in K \}$ for any $x\in\Ren$.
We denote the \emph{polar} of a convex body $K$ by
\[
 K^\ast=\{x\in\Ren\st <x,k>\leq 1 \mbox{ for all } k\in K\}.
\]
The cardinality of a set $X$ is denoted by $\card{X}$.

\section{Basics}\label{sec:basics}

We list a number of simple properties of covering numbers, their proofs 
are quite straight forward, cf.\cite{AGM15}.
\begin{fact}
Let $K,L,M$ be convex sets in $\Ren$, $T:\Ren\to\Ren$ an invertible linear 
transformation. The we have
\begin{eqnarray}
 N(K,L)&=&N(T(K),T(L)),\\
 N(K,L)&\leq& N(K,M)N(M,L),\\
 N(K+M,L+M)&\leq& N(K,L),\\
 N(K,2(K\cap L))&\leq& N(K,L),\mbox{ if } K=-K. \label{eq:twoacapb}
\end{eqnarray}
\end{fact}

We note a special property of the Euclidean ball as a covering set. 
\begin{fact}
Let $K$ be a convex set in $\Ren$. If $K$ is covered by $t$ Euclidean balls, 
then $K$ is covered by $t$ Euclidean balls with centers in $K$.
\end{fact}
This fact follows from the observation that the intersection of $\B_2^n$ with a 
half-space not containing the origin is contained in the unit ball centered at 
the orthogonal projection of the origin to the bounding hyperplane of the 
half-space.

The following obvious lower bound is often sufficient:
\begin{equation}
N(K,L) \geq \frac{\vol(L)}{\vol(K)}.
\end{equation}

For an upper bound, assume that $L$ is symmetric. Let 
$X+(L/2)$ be a \emph{saturated packing} of translates of 
$L/2$ in $K+(L/2)$, that is a maximal family of translates of $L/2$ with 
pairwise disjoint interiors. Then $K\subseteq X+L$, that is, we have a covering 
of $K$ 
by $\card{X}$ translates of $L$. Thus,
\begin{equation}\label{eq:simplepackingbound}
 N(K,L)\leq \card{X}\leq 2^n \frac{\vol(K+(L/2))}{\vol(L)}, \mbox{ if } L=-L.
\end{equation}

Section~\ref{sec:wholespace}, and a large part of this paper discuss how this 
bound can be improved.

\section{Covering the whole space}\label{sec:wholespace}

Let $K$ be a convex body, $\Lambda$ a lattice, and $T$ a finite set in $\Ren$. 
We call the family $\FF=K+\Lambda+T=\{K+v+t\st v\in\Lambda, t\in T\}$ a 
\emph{periodic arrangement} of translates of $K$.
The \emph{density} of $\FF$ is defined as $\delta(\FF)=|T|\vol(K)/\det\Lambda$. 
We say that $\FF$ is a \emph{covering} of $\Ren$ if $\cup\FF=\Ren$. The 
\emph{periodic translative covering density} $\theta(K)$ of $K$ is the infimum 
of the densities of periodic coverings of $\Ren$ by translates of $K$.

The first milestone in the theory of translative coverings is the 
following theorem of Rogers.
\begin{thm}[Rogers, \cite{Ro57}]\label{thm:Rogers}
  Let $K$ be a bounded convex set in $\Ren$ with non-empty interior. Then the 
periodic translative covering density of $K$ is at most
  \begin{equation}
  \theta(K)\leq n\ln n+ n\ln\ln n + 5n.
  \end{equation}
\end{thm}

One can define the translative covering density in general, for non-periodic 
arrangements as well (cf. \cites{LFT53,RoBook64,PaAg95,Bo04}). However 
--perhaps, not 
surprisingly--, no better bound is 
known in this case.

Earlier, exponential upper bounds for the covering density were obtained by 
Rogers, Bambah and Roth, and for the special case of the Euclidean ball by 
Davenport and Watson (cf. \cite{Ro57} for references).
The last summand, $5n$ may be replaced by $3n$, if $n$ is sufficiently large.
The current best bound on $\theta(K)$ is due to G. Fejes T\'oth \cite{FTG09}, 
who replaced $5n$ by $n+o(n)$ (see Theorem~\ref{thm:ftg}). It is an open 
problem whether one can improve the bound by a multiplicative factor bellow 1, 
or, very ambitiously, if $Cn$ is an upper bound, for some universal $C>0$.

It is natural to ask what happens if the density is replaced by the maximum 
multiplicity.
\begin{thm}[Erd\H os, Rogers, \cite{ER61}]\label{thm:ErdosRogers}
For any convex body $K$ in $\Ren$ there is a periodic covering of 
$\Ren$ by translates of $K$ such that no point is covered by more than 
$e(n\ln n+ n\ln\ln n + 4n)$ translates, and the density is below $n\ln n+ 
n\ln\ln n + 4n$, provided $n$ is large enough.
\end{thm}

A good candidate for a ``bad'' convex body, that is, one that cannot cover the
space economically is the Euclidean ball, $\B^n_2$.

\begin{thm}[Coxeter, Few, Rogers, \cite{CFR59}]\label{thm:cfr59}
 $\theta(\B^n_2)\geq Cn$ with a universal constant $C>0$.
\end{thm}

If we restrict ourselves to \emph{lattice coverings}, that is, coverings of 
$\Ren$ by translates of a convex body $K$ where the translation vectors form a 
lattice in $\Ren$, we have a much weaker bound. Rogers \cite{Ro59} showed that 
for any $K$ there is a lattice $\Lambda$ such that $K+\Lambda$ is a covering of 
$\Ren$ with density at most $n^{\log_2\ln n + O(1)}$.

The original proofs of Theorems~\ref{thm:Rogers} and \ref{thm:ErdosRogers} 
yield 
periodic coverings without any further structure. G. Fejes T\'oth gave a proof 
of Theorem~\ref{thm:Rogers} that yields a covering with more of a lattice--like 
structure, and a slightly better density bound.

\begin{thm}[G. Fejes T\'oth, \cite{FTG09}]\label{thm:ftg}
For any convex body $K$ in $\Ren$ there is a lattice $\Lambda$ and a set 
$T\subset\Ren$ of $O(\ln n)$ translation vectors such that $K+\Lambda+T$ covers 
$\Ren$ with density at most $n\ln n + n\ln\ln n + n + o(n)$.
\end{thm}

We give an outline of the proof of this result in \ref{subsec:ftg}.

The following is a simple corollary to Theorem~\ref{thm:Rogers} (or the 
better bound, Theorem~\ref{thm:ftg}), which was first spelled out in 
\cite{RZ97}.
\begin{cor}[Rogers and Zong \cite{RZ97}]\label{cor:rozo}
 Let $K$ and $L$ be convex bodies in $\Ren$. Then
 \begin{equation}
  N(K,L)\leq \frac{\vol(K-L)}{\vol(L)}(n\ln n + n\ln\ln n + n + o(n)).
 \end{equation}
\end{cor}
Indeed, consider a covering $L+G$ of a large cube $C$ by translates of $L$ with 
density close to $n\ln n + n\ln\ln n + n + o(n)$. For any $t\in\Ren$, let 
$m(t)=|\{g\in G\st K\cap (g+t+L)\neq\emptyset\}|=|G\cap(K-L-t)|$. By averaging 
$m(t)$ over $t$ in $C$, we obtain that for some $t\in C$, we have 
$m(t)\leq\frac{\vol(K-L)}{\vol(L)}(n\ln n + n\ln\ln n + n + o(n)+\varepsilon)$.

\section{Proofs of Theorems~\ref{thm:Rogers}, \ref{thm:ErdosRogers} and 
\ref{thm:ftg}
}\label{sec:rogersproofs}

\subsection{A probabilistic proof: Cover randomly and then mind the 
gap}\label{sec:prob}
We give an outline of Rogers' proof of Theorem~\ref{thm:Rogers}.

We may assume that $K$ has volume one, and that the centroid (that is, the 
center of mass with respect to the Lebesgue measure) of $K$ is the origin. It 
follows that $K\subset -n K$. (Bonnesen and Fenchel in \S 34. of \cite{BF87} 
give several references to this fact: Minkowski \cite{Mi1897} p. 105, Radon 
\cite{Ra16}, Estermann \cite{Es28} and S\"uss \cite{Su28}.)

Let $C$ be the cube $C=[0,R]^n$, where $R$ is large. Set $\eta=\frac{1}{n\ln 
n}$, and choose $N=R^nn\ln\frac{1}{\eta}$ random translation vectors 
$x_1,\ldots,x_N$ in $C$ uniformly and independently. Let $\Lambda$ be the 
lattice $\Lambda=R\Ze^n$. Thus, we obtain the family 
$K+\Lambda+\{x_1,\ldots,x_N\}$ of translates of $K$. The 
expected density of the union of this family is close to one, 
and hence, one can choose the $N$ translation vectors in such a way that the 
volume of the uncovered part of $C$ is small (at most $R^n(1-R^{-n})^N$).

Next, we take a saturated (that is, maximal) 
packing $y_1-\frac{1}{n}K,\ldots y_M-\frac{1}{n}K$ of translates of 
$-\frac{1}{n}K$ inside this uncovered part of $C$. By the previous volume 
computation, we have few ($M\leq \eta^{-n} R^n(1-R^{-n})^N$) such 
translates. We replace each of these copies of $-\frac{1}{n}K$ by the same 
translate of $K$ and make it a periodic arrangement by $\Lambda$, and we obtain 
$K+\Lambda+\{y_1,\ldots,y_M\}$. 

Now, we have two families of translates of $K$. We enlarge each member of these 
two families by a factor $1+\eta$, and --as it is easy to see-- obtain a 
covering 
of $\Ren$. The omitted computations yield the density bound, finishing the 
proof of Theorem~\ref{thm:Rogers}.

This method (first, picking random copies, and then, filling the small gap in 
a greedy way), developed by Rogers can be applied for obtaining upper bounds 
in other situations as well. The proof of Theorem~\ref{thm:ErdosRogers} given 
by Erd\H os and Rogers is an example of the use of this random covering 
technique combined with a sophisticated way of keeping track of multiply 
covered points using an inclusion--exclusion formula. Other examples include 
bounds on covering the sphere $\Senm$ with spherical caps.

\subsection{Another probabilistic proof: Using the Lov\'asz Local Lem\-ma}
F\"uredi and Kang \cite{FuKa08} gave a proof of 
Theorem~\ref{thm:ErdosRogers} that is essentially different from the original. 
Their method 
yields a slightly worse bound (instead of the order $en\ln n$, they obtain 
$10n\ln n$), but it is very elegant. 

First, we may assume that $K$ is in a nice position, more precisely (by a 
result of Ball \cite{Ba91}, see also \cite{Ba89} for the symmetric 
case), that $\vol(K)=1$, and $\frac{1}{e}\B_2^n\subset K$.
Let $h=1/(4en\sqrt{n})$, and consider the lattice $\Lambda=h\Ze^n$. The 
goal is to cover $\Ren$ with translates of $K$ of the form $K+z$ with 
$z\in\Lambda$. Let $Q=[0,h)^n$ be the half closed, half open fundamental cube 
of $\Lambda$. We define a hypergraph with base set $\Lambda$. The hypergraph 
has two types of edges. For any $z\in\Lambda$, we define a ``small edge'' as 
$A^-(z):=\{y\in\Lambda\st y+Q\subset z+K\}$, and a ``big edge'' as 
$A^+(z):=\{y\in\Lambda\st (y+Q)\cap (z+K)\neq\emptyset\}$. Clearly, all big 
edges are of the same size (say $\alpha$), and so are all small edges. One can 
verify that the size of a small edge is at least $\alpha/2$.

Next, to make the problem finite, let $\ell\in\Ze^+$ be an arbitrarily large 
integer. Our goal is to select vectors $z_1,\ldots,z_t\in\Lambda$ in such a way 
that every point of $[-\ell,\ell]^n\cap\Lambda$ is covered by a small edge 
$A^-(z_i)$, and no point of $[-\ell,\ell]^n\cap\Lambda$ is covered by more than 
$10n\ln n$ large edges of the form $A^+(z_i)$. Clearly, that would suffice for 
proving the theorem. We will pick these vectors randomly: select each vector in 
$\{z\in\Lambda\st A^-(z)\cap [-\ell,\ell]^n\neq\emptyset\}$ with probability 
$p$, where $p=e^{-6/5}10n\ln n/\alpha$.

For every point of $[-\ell,\ell]^n\cap\Lambda$, we have two kinds of bad 
events. One is if it is not covered by a small edge, and second, if it is 
covered by too many big edges. Now, we state the main tool of the proof, 
the Lov\'asz Local Lemma (see Alon and Spencer \cite{AS08} for a good 
introduction of it).

\begin{lem}[Lov\'asz Local Lemma, \cites{EL75, Sp77}]\label{lem:LLL}
  Let $A_1, A_2,\ldots, A_N$ be events in an arbitrary probability space. A 
directed graph $D=(V,E)$ on the set
of vertices $V=\{1,2,\ldots,N\}$ is called a \emph{dependency digraph} for the 
events $A_1, A_2,\ldots, A_N$, if for each $1\leq i\leq N$, the
event $A_i$ is mutually independent of all the events $\{A_j\st (i,j)\notin 
E\}$.
Suppose that the maximum degree of $D$ is at most
$d$, and that the probability of each $A_i$ is at most $p$. 
If $ep(d+1)\leq 1$, then with positive probability no $A_i$ holds.
\end{lem}

Finally, with a geometric argument, one can bound the maximum degree in a 
dependency digraph of the bad events, and Lemma~\ref{lem:LLL} yields the 
existence of a good covering.

\subsection{Covering using few lattices}\label{subsec:ftg}
G. Fejes T\'oth's proof of Theorem~\ref{thm:ftg} relies on a deep 
result, Theorem~10* in \cite{Sch59} of Schmidt. A consequence of this result is 
\begin{lem}\label{lem:SchCor}
 Let $c_0 = 0.278\ldots$ be the root of the equation $1 + x + \ln x = 0$. Then, 
for any $0<c<c_0$, and $\varepsilon>0$, and any sufficiently large $n$, 
for any Borel set $S\subset\Ren$, there is a lattice--arrangement of $S$ with 
density $cn$ covering $\Ren$ with the exception of a set whose density is at 
most $(1+\varepsilon)e^{-cn}$ for some universal constant $c>0$.
\end{lem}

By this lemma, for a given $K$ there is a lattice $\Lambda$ such that 
$(1+\lfloor n\ln n\rfloor^{-1})^{-1}K+\Lambda$ covers $\Ren$ with the exception 
of a set whose density is at most $e^{-cn+1}$.

\begin{lem}\label{lem:squaredensity}
 If, for some finite set $T$, $K+\Lambda+T$ is an arrangement of $K$ with 
density $1-\delta$, then there is a vector $t\in\Ren$ such that the arrangement 
$K+\Lambda+T^\prime$ has density at least $1-\delta^2$, where 
$T^\prime=T\cup(T+t)$.
\end{lem}

The proof of Lemma~\ref{lem:squaredensity} relies on considering the density of 
$\Ren\setminus(K+\Lambda+T^\prime)$ as a function of $t$, and averaging it over 
the fundamental domain of $\Lambda$.

To prove Theorem~\ref{thm:ftg}, we pick an appropriate $c$ for 
Lemma~\ref{lem:SchCor}, and using
Lemma~\ref{lem:squaredensity} roughly $\log_2(c^{-1}\ln n)$ times, we obtain a 
finite set $T$ of size about $c^{-1}\ln n$ such that $(1+\lfloor n\ln 
n\rfloor^{-1})^{-1}K+\Lambda+T$ has density about $n\ln n$ with the uncovered 
part being of density at most of order $(n\ln n)^{-n}$. Finally, one can verify 
that $K+\Lambda+T$ is a covering of space with the desired density.

So far, we presented three probabilistic methods that yield economical 
coverings. In the next section, we present a fourth method, which is not 
random. Instead, it relies on an algorithmic combinatorial result.

\section{A fractional approach}\label{sec:fractional}

\subsection{A few words of combinatorics}\label{subsec:combinatorics}
We recall some notions from the theory of hypergraphs.

\begin{defn}\label{defn:fraccovgen}
 Let $\Lambda$ be a set, $\HH$ a family of subsets of $\Lambda$. A 
\emph{covering} of $\Lambda$ by $\HH$ is a subset of $\HH$ whose union is 
$\Lambda$. 
The \emph{covering number} $\tau(\Lambda,\HH)$ of $\Lambda$ by $\HH$ is the 
minimum cardinality of its coverings by $\HH$.

A \emph{fractional covering} of $\Lambda$ by $\HH$ is a measure $\mu$ on $\HH$ 
with 
\[
\mu(\{H\in\HH\st p\in H\})\geq 1\;\;\;\mbox{ for all } p\in \Lambda.
\]
The \emph{fractional covering number} of $\HH$ is
\[
 \tau^\ast(\Lambda,\HH)=\inf\left\{\mu(\HH)\st \mu \mbox{ is a fractional 
covering of 
} \Lambda \mbox{ by } \HH\right\}.
\]
\end{defn}

When $\Lambda$ is a finite set, finding the value of $\tau(\Lambda,\HH)$ is an 
integer programming problem. Indeed, we assign a variable $x_H$ to each member 
$H$ of $\HH$, and set $x_H$ to 1 if $H$ is in the covering, and 0 otherwise. 
Each element $p$ of $\Lambda$ yields an inequality: $\sum_{p\in H\in\HH}x_H\geq 
1$.

Computing $\tau^\ast(\Lambda,\HH)$ is the linear relaxation of the above 
integer programming problem. For more on (fractional) coverings, cf. 
\cite{Fu88} in the abstract (combinatorial) setting and \cites{PaAg95,Ma02} in 
the geometric setting.

The gap between $\tau$ and $\tau^\ast$ is bounded in the case of finite set 
families (hypergraphs) by the following result of Lovász \cite{Lo75} and 
Stein\cite{St74}.

\begin{lem}[Lovász \cite{Lo75}, Stein\cite{St74}]\label{lem:Lovasz}
 For any finite $\Lambda$ and $\HH\subseteq 2^\Lambda$ we have
  \begin{equation}\label{eq:LovaszIG}
    \tau(\Lambda,\HH) < (1+\ln(\max_{H\in \HH}\card{H}))\tau^\ast(\Lambda,\HH).
  \end{equation}
  Furthermore, the greedy algorithm (always picking the set that covers the 
largest number of uncovered points) yields a covering of cardinality less than 
the 
right hand side in \eqref{eq:LovaszIG}.
\end{lem}

We note that a probabilistic argument yields a slightly different bound on the 
covering number:
\begin{equation}\label{eq:probtaubound}
    \tau(\Lambda,\HH) \leq \left\lfloor 1+\frac{\ln 
\card{\Lambda}}{-\ln\left(1-\frac{1}{\tau^\ast}\right)} 
\right\rfloor,
\end{equation}
with the notation $\tau^\ast=\tau^\ast(\Lambda,\HH)$. When we do not have an 
upper bound on $\max_{H\in \HH}\card{H}$ better than $\card{\Lambda}$, then 
\eqref{eq:probtaubound} is a bit better than \eqref{eq:LovaszIG}.

To prove \eqref{eq:probtaubound}, let $\mu$ be fractional covering of $\Lambda$ 
by $\HH$ such that $\mu(\HH)=\tau^\ast+\varepsilon$, where $\varepsilon>0$ is 
very small. We normalize $\mu$ to obtain the probability measure 
$\nu=\mu/\mu(\HH)$ on $\HH$. Let $m$ denote the right hand side in 
\eqref{eq:probtaubound}, and pick $m$ members of $\HH$ randomly according to 
$\nu$. Then we have
\[
 \Pe\left(\exists u\in\Lambda\st u \mbox{ is not covered}\right)\leq
 \card{\Lambda}\left(1-\frac{1}{\tau^\ast+\varepsilon}\right)^m<1.
\]
Thus, with positive probability, we have a covering.

We will need the duals of these notions as well. Let $\Lambda$ be a set and 
$\HH$ be 
a family of subsets of $\Lambda$. The \emph{dual} of this set family is another 
set 
family, whose base set is $\HH$, and the set family on $\HH$ is 
$\HH^\ast=\big\{\{H\in\HH\st p\in H\}\st p\in \Lambda\big\}$. 

We call a set $T\subset \Lambda$ a \emph{transversal} to the set family $\HH$, 
if $T$ 
intersects each member of $\HH$. One may define \emph{fractional transversals} 
in the obvious way, and then define the \emph{(fractional) transversal number}.

Clearly $\GG\subset\HH$ is a covering of $\Lambda$ if and only if, $\GG$ is a 
transversal to $\HH^\ast$. Fractional coverings and fractional transversals are 
dual notions in the same manner. We leave it as an exercise (which will be 
needed later) to formulate the dual of Lemma~\ref{lem:Lovasz} and of 
\eqref{eq:probtaubound}.

\subsection{The fractional covering number}\label{subsec:fraccovvnumber}

Motivated by the above combinatorial notions, the fractional 
version of $N(K,\inter K)$ (which is the illuminaton number of $K$, see 
Section~\ref{sec:illum}) first appeared in \cite{Na09}, and in general for 
$N(K,L)$ in \cite{AR11} and \cite{AS}.

\begin{defn}\label{defn:fraccovcvx}
 Let $K$ and $L$ be bounded Borel measurable sets in $\Ren$. A \emph{fractional 
covering} of $K$ by translates of $L$ is a Borel measure $\mu$ on $\Ren$ with 
$\mu(x-L)\geq 1$ for all $x\in K$. The \emph{fractional covering number} of $K$ 
by translates of $L$ is
\[
 N^\ast(K,L)=
 \]\[\inf\left\{\mu(\Ren)\st \mu \mbox{ is a fractional covering of } K 
\mbox{ by translates of } L\right\}.
\]
\end{defn}

Clearly,
\begin{equation}\label{eq:obviousfracbound}
N^\ast(K,L)\leq N(K,L).
\end{equation}

In Definition~\ref{defn:fraccovcvx} we may assume that a fractional 
cover $\mu$ is supported on $\cl(K-L)$. 
According to Theorem~1.7 of \cite{AS}, we have
\begin{equation}\label{eq:simpleupperbound}
\max\left\{\frac{\vol(K)}{\vol(L)},1\right\}\leq 
N^\ast(K,L)\leq
\frac{\vol(K-L)}{\vol(L)}.
\end{equation}

The second inequality is easy to see: the 
Lebesgue measure restricted to $K-L$ with the following scaling 
$\mu=\vol/\vol(L)$ is a fractional covering of $K$ by translates of $L$. To 
prove the first inequality, assume that $\mu$ is a fractional covering of $K$ 
by translates of $L$. Then
\[
 \vol(L)\mu(\Ren)=\int_{\Ren} \vol(L)\di 
\mu(x)=\int_{\Ren}\int_{\Ren}\chi_{L}(y-x)\di y\di\mu(x)=
\]\[
 \int_{\Ren}\int_{\Ren} \chi_{L}(y-x)\di\mu(x)\di y=
\int_{\Ren}\mu(y-L) \di y\geq
 \int_{\Ren}\chi_K(y) \di y=\vol(K).
\]

We recall from \ref{subsec:combinatorics} that computing $N$ means solving an 
integer programming problem (though, in this situation, with infinitely many 
variables), and 
computing $N^\ast$ is its linear relaxation. The linear relaxation is 
usually easier to solve, so having an inequality bounding $N$ from above by 
some function of $N^\ast$ is desirable. It is open whether such inequality 
exists in general for convex sets. More precisely, we do not know if there is a 
function $f$ such that for any dimension $n$, and any convex bodies $K$ and $L$ 
in $\Ren$, we have $N(K,L)\leq f(n, N^\ast(K,L))$.

Using a probabilistic argument, Artstein--Avidan and Slomka \cite{AS} found a 
bound of $N(K,L)$ in terms of $N^\ast(K^\prime,L^\prime)$, where $K^\prime$ and 
$L^\prime$ are almost $K$ and $L$. A somewhat stronger bound was obtained in 
\cite{N15} by a non-probabilistic proof. 
For two sets $K, T\subset\Ren$, we denote their \emph{Minkowski difference} by 
$K\sim T=\{x\in \Ren\st T+x \subseteq K\}$.

\begin{thm}[Artstein--Avidan and Slomka \cite{AS}, Nasz\'odi 
\cite{N15}]\label{thm:cvxIG}
  Let $K, L$ and $T$ be bounded Borel measurable sets in $\Ren$ and let 
$\Lambda\subset\Ren$ be a finite set with $K\subseteq \Lambda+T$. Then
  \begin{equation}\label{eq:cvxIG}
  N(K,L)\leq 
  \end{equation}
\[ 
  (1+\ln(
  \max_{x\in K-L} \card{(x+(L\sim T))\cap \Lambda }
  ))
  \cdot N^\ast(K-T,L\sim T).
\]
  If $\Lambda\subset K$, then we have
  \begin{equation}\label{eq:cvxIGspec}
  N(K,L)\leq 
  \end{equation}
\[  (1+\ln(
  \max_{x\in K-L} \card{(x+(L\sim T))\cap \Lambda }
  ))
  \cdot N^\ast(K,L\sim T).
\]
\end{thm}

We sketch a proof of Theorem~\ref{thm:cvxIG} in \ref{subsec:fracproof}.

For a set $K\subset\Ren$ and $\delta>0$, we denote the \emph{$\delta$-inner 
parallel body} of $K$ by $K_{-\delta}:=K\sim \delta\B^n_2=\{x\in K\st 
x+\delta\B^n_2\subseteq K\}$. As an application of Theorem~\ref{thm:cvxIG}, 
one quickly obtains the following result which, in turn, may be used to give a 
simple proof of Rogers' result, Theorem~\ref{thm:Rogers}.

\begin{thm}[Nasz\'odi \cite{N15}]\label{thm:Renbyanything}
Let $K\subseteq\Ren$ be a bounded measurable set.
Then there is a covering of $\Ren$ by translated copies of $K$ of density at 
most
\[
 \inf_{\delta>0}\left[
  \frac{\vol(K)}{\vol(K_{-\delta})}
  \left( 
1+\ln\frac{\vol\left(K_{-\delta/2}\right)}{\vol\left(\frac{\delta}{2}
\B^n_2\right)
}\right)\right].
\]
\end{thm}
The theorem still holds if the $\delta$-inner parallel body is defined with 
respect to any norm in place of the Euclidean.

\subsection{Proof of Theorem~\ref{thm:cvxIG}}\label{subsec:fracproof}
The proofs outlined so far were all probabilistic in nature. In this one, the 
role that  probability plays elsewhere is played by the following 
straightforward corollary to Lemma~\ref{lem:Lovasz}.

\begin{obs}\label{obs:IG}
Let $Y$ be a set, $\FF$ a family of subsets of $Y$, and $X\subseteq Y$. 
Let $\Lambda$ be a finite subset of $Y$ and $\Lambda\subseteq U\subseteq Y$. 
Assume 
that for another family $\FF^\prime$ of subsets of $Y$ we have 
$\tau(X,\FF)\leq\tau(\Lambda,\FF^\prime)$. Then
  \begin{equation}\label{eq:IG}
  \tau(X,\FF)\leq 
  \tau(\Lambda,\FF^\prime)\leq
  (1+\ln( \max_{F^\prime\in\FF^\prime} \card{\Lambda\cap F^\prime} ) ) 
  \cdot \tau^\ast(U, \FF^\prime).
  \end{equation}
\end{obs}

The proof is simply a substitution into \eqref{eq:IG}. We set $Y=\Ren$, $X=K$, 
$\FF=\{L+x\st x\in K-L\}$, $\FF^\prime=\{L\sim T+x\st x\in K-L \}$. One can 
use $U=K-T$, as any member of $\Lambda$ not in $K-T$ could be 
dropped from $\Lambda$ and $\Lambda$ would still have the property that 
$\Lambda +T \supseteq K$. That proves \eqref{eq:cvxIG}. To prove 
\eqref{eq:cvxIGspec}, we notice that in the case when $\Lambda\subset K$, 
one can take $U=K$.

\subsection{Detour: Covering the sphere by caps}
To illustrate the applicability of the method that yields 
Theorem~\ref{thm:cvxIG}, we turn to coverings on the sphere.
We denote the closed spherical cap of spherical radius $\phi$ centered at 
$u\in\Senm$ by $C(u,\phi)=\{v\in\Senm\st <u,v>\geq\cos\phi\}$, and its 
probability measure by $\Omega(\phi)=\sigma(C(u,\phi))$. For a set 
$K\subset\Senm$ and $\delta>0$, we denote the \emph{$\delta$--inner parallel 
body} of $K$ by $K_{-\delta}=\{u\in K\st C(u,\delta)\subseteq K\}$.

A set $K\subset\Senm$ is called \emph{spherically convex}, if it is contained 
in 
an open hemisphere and for any two of its points, it contains the shorter great 
circular arc connecting them.

The \emph{spherical circumradius} of a subset of an open hemisphere of $\Senm$ 
is the spherical radius of the smallest spherical cap (the \emph{circum-cap}) 
that contains the set. A proof mimicking the proof of Theorem~\ref{thm:cvxIG} 
yields

\begin{thm}[Nasz\'odi \cite{N15}]\label{thm:spherebyanything}
Let $K\subseteq\Senm$ be a measurable set.
Then there is a covering of $\Senm$ by rotated copies of $K$ of density at most
\[
 \inf_{\delta>0}\left[
  \frac{\sigma(K)}{\sigma(K_{-\delta})}
  \left( 
1+\ln\frac{\sigma\left(K_{-\delta/2}\right)}{\Omega\left(\frac{\delta}{2}\right)
}\right)\right].
\]
\end{thm}

Improving an earlier result of Rogers \cite{Ro63}, B\"or\"oczky and Wintsche 
\cite{BW03} showed that for any $0<\varphi<\pi/2$ and dimension $n$ there is a 
covering of $\Sen$ by spherical caps of radius $\phi$ with density at most 
$n\ln n+n\ln\ln n+5n$. This result follows from 
Theorem~\ref{thm:spherebyanything}. Other bounds on covering the sphere by caps 
(or, a ball by smaller equal balls) can be found in \cite{Ve05} by 
Verger--Gaugry.

\section{The Illumination Conjecture}\label{sec:illum}

We fix a convex body $K$ in $\Ren$.
Once the covering number is defined, it is fairly natural to ask what Levi 
\cite{Lev55} asked: how large may $N(K,\inter K)$ be.  We 
will call this quantity the \emph{illumination number} of $K$, and denote it by 
$i(K)=N(K,\inter K)$. The naming will become obvious in the next paragraphs.

Following Hadwiger \cite{Ha60}, we say that a point $p\in\Ren\setminus K$ 
\emph{illuminates} a boundary point $b\in\bd K$, if the ray $\{p+\lambda 
(b-p)\st \lambda>0\}$
emanating from $p$ and passing through $b$ intersects the interior of $K$.
Boltyanski \cite{Bo60} gave the following slightly different definition. A 
direction $u\in\Se^{n-1}$ is said to 
\emph{illuminate} $K$ at a boundary point $b\in\bd K$, if the ray $\{b+\lambda 
u\st \lambda>0\}$ intersects the interior of $K$.
It is easy to see that the minimum number of directions 
that illuminate each boundary point of $K$ is equal to the minimum number of 
points that illuminate each boundary point of $K$, which in turn is equal to 
the illumination number of $K$. 

Gohberg and Markus \cite{GoMa60} asked how large $\inf\{N(K,\lambda K\st 
0<\lambda<1)\}$ can be. It also follows easily that this number is equal to 
$i(K)$.

The following dual formulation of the definition of the illumination number was 
found independently by P. Soltan, V. Soltan \cite{SolSol86} and by Bezdek 
\cite{Be91}. 
First, recall that an \emph{exposed face} of a convex body $K$ is the 
intersection of $K$ with a supporting hyperplane.
Now, let $K$ be a convex body in $\Ren$ containing the origin in its 
interior. Then $i(K)$ is the minimum size of a family of hyperplanes in $\Ren$ 
such that each exposed face of the polar $K^\ast$ of $K$ is strictly separated 
from the origin by at least one of the hyperplanes in the family (for the 
definition of $K^\ast$, see the introduction).

Any smooth convex body (ie., a convex body with a unique support hyperplane at 
each boundary point) in $\Ren$ is illuminated by $n+1$ 
directions. Indeed, for a smooth convex body, the set of directions 
illuminating a 
given boundary point is an open hemisphere of $\Senm$, and one can find 
$n+1$ points (eg., the vertices of a regular simplex) in $\Senm$ with 
the property that every open hemisphere contains at least one of the points. 
Thus, these $n+1$ points in $\Senm$ (ie., directions) illuminate any smooth 
convex body in $\Ren$. It is easy to see that no convex body is illuminated by 
less than $n+1$ directions.

On the other hand, the illumination number of the cube is $2^n$, since no two 
vertices of the cube share an illumination direction. An important unsolved 
problem in Discrete Geometry is the 
\emph{Gohberg--Markus--Levi--Boltyanski--Hadwiger Conjecture} (or, Illumination 
Conjecture), according to which \emph{for any convex body $K$ in $\Ren$, we 
have $i(K)=2^n$, where equality is attained only when $K$ is an affine image of 
the cube.}

In this section, we mention some results on illumination. For a more complete 
account of the current state of the problem, see \cites{Be06, Be10, BMP05, 
MS99, Sza97}. In Chapter VI. of \cite{BoMaSo97}, among many other facts on 
illumination, one can find a proof of the equivalence of the first four 
definitions of $i(K)$ given at the beginning of this section. Quantitative 
versions of the illumination number are discussed in Chapter {\bf 
\color{red}[Refernce to Bezdek--Kahn chapter]} of this volume.

One detail of the history of the conjecture may tell a lot about it. It was 
asked several times in different formulations (see the different definitions of 
$i(K)$ above), first in 1960 (though, Levi's study of $N(K,\inter K)$ on the 
plane is from 1955). Several partial results appeared solving the 
conjecture for special families of convex bo\-di\-es. Yet, the best general 
bound 
is an immediate consequence of Rogers' Theorem~\ref{thm:Rogers} (more 
precisely, Corollary~\ref{cor:rozo}) from 1957 combined with the 
\emph{Rogers--Shepard 
inequality} \cite{RoS57}, according to which 
$\vol(K-K)\leq\binom{2n}{n}\vol(K)$ 
for any convex body $K$ in $\Ren$.
\begin{thm}[Rogers \cite{Ro57}]
Let $K$ be a convex body in $\Ren$. Then
 \[
i(K)\leq \begin{cases} 
2^n(n\ln n+ n\ln\ln n + 5n) &\mbox{if } K=-K, \\
\binom{2n}{n}(n\ln n+ n\ln\ln n + 5n) & \mbox{otherwise.}
\end{cases} 
\]
\end{thm}

By \cite{Lev55}, the Illumination Conjecture holds on the plane. Papadoperakis 
\cite{Pap99} proved $i(K)\leq16$ in dimension three. The upper bound in the 
conjecture (that is, not the equality case) was verified in the following 
cases: if $K=-K\subset\Re^3$ (Lassak \cite{Las84}), if $K\subset\Re^3$ is a 
convex polyhedron with at least one non-trivial affine symmetry (Bezdek 
\cite{Be91}), if $K\subset\Re^3$ is symmetric about a plane (Dekster 
\cite{Dek00}).

\subsection{Borsuk's problem and illuminating sets of constant width 
}\label{subsec:borsuk}

The problem of illumination is closely related to another classical question in 
geometry. \emph{Borsuk's problem} \cite{Bor33} (or, Borsuk's Conjecture, 
though, he formulated it as a question) asks whether every bounded set $X$ in 
$\Ren$ can be partitioned into $n+1$ sets of diameter less than the diameter of 
$X$ (cf. \cite{Rai08} for a comprehensive survey). The minimum 
number of such parts is the \emph{Borsuk number} of $X$, and 
clearly, it is at most the illumination number of $\conv(X)$. Since any bounded 
set in $\Ren$ is contained in a set of constant width of the same diameter, it 
follows that any upper bound on the illumination number of sets of constant 
width in a certain dimension is also a bound on the maximum Borsuk number in 
the same dimension. 

The affirmative answer to Borsuk's problem in the plane was proved by Borsuk, 
then, in three--space by Perkal \cite{Per47} and Eggleston \cite{Egg55} (in the 
case of finite, three-dimensional sets, see Gr\"unbaum \cite{Gru57}, 
Heppes--R\'ev\'esz \cite{HepRev56} and Heppes \cites{Hep57}).
It was first shown by Lassak \cite{Las97} (see also \cites{Wei96, 
BLNP07}) that sets of constant width in $\Re^3$ can be illuminated by three 
pairs of opposite directions. It would be a nice alternative proof of the bound 
4 on the 
Borsuk number in three--space, if one could show that three--dimensional sets 
of 
constant width have illumination number 4 (see Conjecture~3.3.5. in 
\cite{Be10}).

In 1993 by an ingenious proof, Kahn and Kalai \cite{KahKal93} 
(based on a deep combinatorial result of Frankl and Wilson \cite{FraWil81}) 
showed that if $n$ is large enough, then there is a finite set in 
$\Ren$ whose Borsuk number is greater than $(1.2)^{\sqrt{n}}$, thus answering 
Borsuk's question in the negative. That result made the following bound on 
the illumination number by Schramm \cite{Sch88} all the more relevant. 
Currently, this is also the best general bound for the Borsuk number.
\begin{thm}[Schramm \cite{Sch88}]\label{thm:schramm}
 In any dimension $n$ for any set $W$ of constant width in $\Ren$, we have
 \[
  i(W)\leq 5n\sqrt{n}(4+\ln n)\left(\frac{3}{2}\right)^{n/2}.
 \]
\end{thm}
By a fine analysis of Schramm's method, Bezdek (Theorem~6.8.3. of \cite{Be10}) 
extended Theorem~\ref{thm:schramm} to the class of those convex bodies $W$ that 
can be obtained as $W=\cap_{x\in X} (x+\B_2^n)$ for some $X\subset\Ren$ compact 
set with $\diam X\leq 1$. Note that a set $W$ is of constant width one if and 
only if, $W=\cap_{x\in W} (x+\B_2^n)$.

We sketch the proof. First, we give yet another way to compute the 
illumination number of a convex body $K$. Let $b$ be a boundary point 
of $K$, and consider its \emph{Gauss image} $\beta(b)\subset\Senm$ consisting 
of the 
inner unit normal vectors of all hyperplanes supporting $K$ at $b$. It is a 
closed, spherically convex set. We denote the \emph{open polar} of a subset of 
the sphere $F\subset\Senm$ by $F^+=\{u\in\Senm\st <u,f> > 0 \mbox{ for all } 
f\in F\}$.
Consider the set family $\FF=\{(\beta(b))^+\st b\in\bd K\}$. Clearly, the
directions $u_1,\ldots,u_m\in\Senm$ illuminate $K$ if and only if, each member 
of $\FF$ contains at least one $u_i$. In other words, we are 
looking for a small cardinality transversal to the set family $\FF$ (for 
definitions, see \ref{subsec:combinatorics}).
We note that the idea of considering the Gauss image and $\FF$ to bound 
the illumination number also appears in \cites{Be91,BLNP07,BeKi09}.

Now, consider a set $W=\cap_{x\in X} (x+\B_2^n)$ with a compact set 
$X\subset\Ren$ of diameter at most one. To make the 
problem of bounding $i(W)$ finite, we take a covering of $\Senm$ by spherical 
caps of Euclidean 
diameter $\varepsilon:=\sqrt{\frac{2n}{2n-1}}-1$, say $C_1\cup\ldots\cup 
C_N=\Senm$. 
Such covering exists with $N\leq(1+\frac{4}{\varepsilon})^n$ by the simple 
bound \eqref{eq:simplepackingbound}. We could use a better bound, but that 
would not yield any visible improvement on the bound on $i(W)$.
Let 
\[
 U_i:=\bigcup_{\beta(b)\cap C_i\neq\emptyset}\beta(b),
\]
and consider the set family $\GG=\{U_i^+\st i=1,\ldots,N\}$. Clearly, 
any transversal to the finite set family $\GG$ is a transversal to $\FF$, and 
hence, is a set that illuminates $K$. One can show that
\begin{equation}\label{eq:gaussdiam}
 \diam(U_i)\leq 1+\varepsilon.
\end{equation}
Let $V(t):=\inf\{\sigma(F^+)\st F\subset\Senm, \diam S\leq t\}$. A key element 
of the proof is the highly non-trivial claim that
\begin{equation}\label{eq:Vboundedbellow}
 V(t)\geq \frac{1}{\sqrt{8\pi 
n}}\left(\frac{3}{2}+\frac{\left(2-\frac{1}{n}\right)t^2-2}{4-\left(2-\frac{2}{n
}\right)t^2} \right)^{-\frac{n-1}{2}}
\end{equation}
for all $0<t<\sqrt{\frac{2n}{n-1}}$ and $n\geq 3$.

We notice that by \eqref{eq:gaussdiam}, $\frac{\sigma}{V(1+\varepsilon)}$ is a 
fractional transversal to $\GG$.
Now, the original proof is completed by applying the dual of 
\eqref{eq:probtaubound} to get $i(W)\leq\left\lfloor1+\frac{\ln 
N}{-\ln(1-V(1+\varepsilon))}\right\rfloor$. Substituting the bound on $N$ and 
\eqref{eq:Vboundedbellow}, the theorem follows. Another way to complete the 
proof is to use the dual of Lemma~\ref{lem:Lovasz}, which yields the slightly 
worse bound 
$i(W)<\frac{1+\ln N}{V(1+\varepsilon)}$.

\subsection{Fractional illumination}
The notion of fractional illumination was defined in \cite{Na09}, and then 
further studied in \cite{AR11}.
\begin{defn}
 The \emph{fractional illumination number} of a convex body $K$ in $\Ren$ is 
 \[
  i^\ast(K)=N^\ast(K,\inter K).
 \]
\end{defn}
It was observed in \cite{Na09} that by \eqref{eq:simpleupperbound} and the 
Rogers--Shepard inequality ($\vol(K-K)\leq\binom{2n}{n}\vol(K)$) we have
\begin{equation}\label{eq:fracbound}
i^\ast(K)\leq \begin{cases} 
2^n &\mbox{if } K=-K, \\
\binom{2n}{n} & \mbox{otherwise.}
\end{cases} 
\end{equation}
The fractional form of the Illumination Conjecture (weaker than the original) 
reads: \emph{$i^\ast(K)\leq2^n$, and equality is attained by parallelotopes 
only.} When $K$ is symmetric, the case of equality was settled by 
Artstein--Avidan and Slomka \cite{AS} using a lemma by Schneider.

Interestingly, no better bound is known, so the fractional form of the 
Illumination Conjecture does not seem much easier than the original. 
On the other hand, just as in general, for $N(K,L)$ and $N^\ast(K,L)$, we do 
not have an upper bound of $i(K)$ in terms of $i^\ast(K)$.

The \emph{fractional version of Borsuk's problem} can be stated in a natural 
way, and was investigated in \cite{HuLa14}. We note that the example of a set 
in $\Ren$ with high Borsuk number given by Kahn and Kalai (see 
\ref{subsec:borsuk}) is a set with high fractional Borsuk number as well.

\section{Decomposability of multiple coverings}\label{sec:decomposability}
An \emph{$m$--fold covering} of $\Ren$ by translates of a set $K$ is a 
family $\FF$ of translates of $K$ such that each point is contained in at least 
$m$ members. It is a natural question whether, for a particular $K$, if $m$ is 
large enough (say, at least $m(K)$), then all $m$--fold coverings of $\Ren$ by 
translates of $K$ can be \emph{decomposed} into two coverings. That is, can 
$\FF$ be colored with two colors such that each color class of $\FF$ is a 
covering of $\Ren$?

It was proved in \cite{Pa86} that if $K$ is a centrally symmetric convex 
polygon then such $m(K)$ exists. This was generalized to all convex polygons in 
\cites{TaTo07, PaTo10}.

Arguably the most natural special case was asked by Pach \cite{Pa80}: consider 
the open unit disk. The un--published manuscript \cite{ManPa86} was cited 
several times as having given a positive answer in this case, though, Pach 
\cite{PaTo09} warned that the result ``has not been independently verified.''
The following result of Mani--Levitska and Pach (see \cite{AS08}) also 
suggested that such $m(K)$ should exist for unit disks. \emph{
For every $n\geq2$, there is a positive constant $c_n$ with the following 
property. For every positive integer $m$, any $m$--fold covering of $\Ren$ with 
unit balls can be decomposed into two coverings, provided that no point of the 
space belongs to more than $c_n2^{m/n}$ balls.} This result was one of the 
first geometric applications of the Lov\'asz local lemma.

Surprisingly, the answer is negative. P\'alv\"olgyi and Pach 
\cite{PaPa15} recently showed that \emph{there is no such $m(K)$ for the open 
unit disk.}

Their proof consists of a combinatorial part followed by an intricate 
geometric argument. First, consider the dual problem in the abstract setting of 
hypergraphs. We fix an $m$, and based on \cite{Pal10}, construct an abstract 
hypergraph $\HH$ with the property that each edge contains at least $m$ 
vertices, but for any two--coloring of the vertices, at least one edge will 
contain only vertices of one color. Then, the hypergraph is given a 
\emph{geometric realization}, that is, the vertex set is mapped to a set of 
points on the plane, and the edges are mapped to open unit disks in an 
incidence--preserving manner, that is, a vertex belongs to an edge if and only 
if, the corresponding point belongs to the corresponding disk. This yields a 
finite set (the centers of the disks) in the plane that is $m$--fold covered by 
unit disks (the disks around the points) in an indecomposable way. Finally, 
this $m$--fold covering of this finite set is extended to an $m$--fold covering 
of the whole plane without adding any disk that contains any of the points in 
the finite set.

For more on decomposability of coverings, see \cite{PaPa15}, and references 
therein.

\section{An asymptotic view}\label{sec:asymptotic}

In this section, we present two topics to illustrate the point of view taken in 
the asymptotic theory of convex bodies on the problem of translative coverings.
%For an introduction to the field, refer to \cites{BGVV14,AGM15}.

\subsection{\texorpdfstring{Sudakov's inequlity \nopunct}{Sudakov's inequlity}} 
relates the minimum number of Euclidean balls that cover a 
symmetric convex body to the \emph{mean width} of the body, where the latter is 
defined as
\begin{equation}
 w(K)=\int_{\Senm} h_K(u)+h_K(-u) \;\;\di\sigma(u).
\end{equation}
(See the definition of $\sigma$ and $h_K$ in the introduction.)

\begin{thm}[Sudakov's inequality, \cite {Su71}]\label{thm:sudakov}
For any symmetric convex body $K$ in $\Ren$ and any $t>0$, we have
\[
 \log N(K,t\B_2^n)\leq cn\left(\frac{w(K)}{t}\right)^2
\]
with an absolute constant $c>0$.
\end{thm}

It was observed by Tomczak--Jaegermann \cite{TJ87}, that this inequality can be 
obtained from a dual form proved by Pajor and Tomczak--Jaegermann \cite{PTJ85}.

\begin{thm}[Dual Sudakov inequality]\label{thm:dualsudakov}
For any symmetric convex body $K$ in $\Ren$ and any $t>0$, we have
\[
 \log N(\B_2^n, tK)\leq cn\left(\frac{w(K^\ast)}{t}\right)^2
\]
with an absolute constant $c>0$.
\end{thm}

First, we sketch a proof of Theorem~\ref{thm:dualsudakov} due to Talagrand 
\cite{Ta91}, \cite{LT91}, and later turn to the proof of 
Theorem~\ref{thm:sudakov}.
\noshow{
  [374 and 108 in A-Avidan, Milman, Giann...]. [318: Ledoux Talgrand]
}
The main idea is to apply a volumetric argument, but, instead of using the 
Lebesgue measure, one uses the Gaussian measure. Recall, that the 
\emph{Gaussian measure} $\gamma_n$ is an absolutely continuous probability 
measure on $\Ren$, with density
\begin{equation*}
 \di\gamma_n(x) = \frac{e^{-|x|^2/2}}{(2\pi)^{n/2}}\;\; \di x.
\end{equation*}
First, by computation one obtains that for any origin--symmetric convex body 
$K$ in $\Ren$ and any translation vector $z\in\Ren$, we have
\begin{equation}\label{eq:gaussiantranslate}
 \gamma_n(K+z)\geq e^{-|z|^2/2}\gamma_n(K).
\end{equation}
Next, we consider a maximal set $\{x_1,\ldots,x_N\}$ in $\B_2^n$ with the 
property that $\|x_i-x_j\|_K\geq t$ for all $i,j$ pairs. Now, for any rescaling 
factor $\lambda>0$, we have that
$\{\lambda x_i+\frac{\lambda t}{2}K\st i=1,\ldots,N\}$ is a packing in 
$\lambda\B^n_2$, and thus, 
the total $\gamma_n$--measure of these sets is at most one. 
Integration in polar coordinates yields that 
\begin{equation*}
 \gamma_n\left(\frac{\lambda t}{2}K\right)\geq 1-\frac{2c\sqrt{n}}{\lambda 
t}w(K^\ast)
\end{equation*}
for an absolute constant $c>0$.
With the choice 
$\lambda=4c\sqrt{n}w(K^\ast)/t$, we have $\gamma_n(K)\geq\frac{1}{2}$. Finally, 
using \eqref{eq:gaussiantranslate}, we obtain the bound in 
Theorem~\ref{thm:dualsudakov}.

We note that this proof yields a little more than stated in the Theorem. We 
obtain an upper bound on the minimum size of a covering of $\B_2^n$ by 
translates of $tK$ with the constraint that the translation vectors are in 
$\B_2^n$.

The following Lemma is the key to reducing Theorem~\ref{thm:sudakov} to 
Theorem~\ref{thm:dualsudakov}. 

\begin{lem}[Tomczak--Jaegermann 
\cite{TJ87}]\label{lem:sudakovdualequivalence}
 For any origin--symmetric convex body $K$ in $\Ren$, and any $t>0$, we have
 \[
  N(K,t\B_2^n)\leq
 N(K,2t\B_2^n)N\left(\B_2^n,\frac{t}{8}K^\ast\right).
 \]
\end{lem}
\begin{proof}[Proof of Lemma~\ref{lem:sudakovdualequivalence}]
Observe that $2K\cap 
\left(\frac{t^2}{2}K^\ast\right)\subseteq t\B_2^n$. Thus, by 
\eqref{eq:twoacapb},
\[
 N(K,t\B_2^n)\leq 
 N\left(K,2K\cap \frac{t^2}{2}K^\ast\right)\leq 
 N\left(K, \frac{t^2}{4}K^\ast\right)\leq 
\]
\[
N(K,2t\B_2^n)N\left(\B_2^n,\frac{t}{8}K^\ast\right).
\]
\end{proof}

\begin{proof}[Proof of Theorem~\ref{thm:sudakov}]
Combining Lemma~\ref{lem:sudakovdualequivalence} and 
Theorem~\ref{thm:dualsudakov}, we have
\[
 t^2\log N(K,t\B_2^n)\leq
 \frac{1}{4}(2t)^2\log N(K,2t\B_2^n)+
 64(t/8)^2\log N\left(\B_2^n,\frac{t}{8}K^\ast\right)
\]
Taking supremum over all $t>0$, we get
\[
 \frac{3}{4} \sup_{t>0} \left\{ t^2\log N(K,t\B_2^n)\right\}\leq
 64\sup_{t>0} \left\{t^2\log N\left(\B_2^n,tK^\ast\right)\right\}
 \leq 64cn\left(w(K)\right)^2.
\]
 
\end{proof}

\subsection{Duality of covering numbers} We briefly mention the 
following open problem in geometric analysis, for a comprehensive 
discussion, cf. Chapter 4 of \cite{AGM15}.
\begin{conj}
There are universal constants $c,C>0$ such that for any dimension $n$ and any 
two symmetric convex bodies $K$ and $L$ in $\Ren$, we have
\[
 N(K,L)\leq N(L^\ast,cK^\ast)^C.
\]
\end{conj}

The problem is known as the \emph{Duality of entropy}, and was posed by Pietsch 
\cite{Pi72}. An important special case, when $K$ or $L$ is a Euclidean ball 
(or, equivalently, an ellipsoid) was confirmed by Artstein--Avidan, Milman and 
Szarek \cite{AMSz04}.

\section{Covering by sequences of sets}\label{sec:sequences}

So far, we considered problems where a set was to be covered by translates of 
another fixed set. Now, we turn to problems where a family $\FF$ of sets is 
given, and we need to find a translation for each set in $\FF$ to obtain a 
covering of a given set $C$. If such translations exist, we say that 
$\FF$ \emph{permits a translative covering} of $C$.
We call $\FF$ a \emph{bounded family}, if the set of diameters of members of 
$\FF$ is a bounded set.

For a comprehensive account of coverings by sequences of convex sets, see the 
surveys \cites{Gr85, FTG04}. Here, we do not discuss a closely related topic, 
Tarski's plank problem -- for a recent survey, see \cite{Be13}.

\subsection{Covering (almost) the whole space.} Clearly, for $\FF$ to permit a 
translative covering of $\Ren$, it is necessary that the total volume of the 
members of $\FF$ be infinite. It is not sufficient, though. Indeed, consider 
rectangles of side lengths $i$ by $1/i^2$ for $i=1,2,\ldots$ on the plane. 
Their total area is infinite, and yet, according to Bang's theorem, they do not 
permit a translative covering of $\Re^2$ \cite{FTG04}. 
On the other hand, if a family of planar convex sets is bounded and has 
infinite total area, then it permits a translative covering of $\Re^2$ 
\cites{MaPa83, Gr82}. It is an open problem whether the same holds for $n>2$.

A covering of \emph{almost all of} some set $C$ is a covering of a subset of 
$C$ whose complement in $C$ is of measure zero.
\begin{thm}[Groemer, \cite{Gr79}]
Let $\FF$ be a bounded family of Lebesgue measurable sets. 
Then $\FF$ permits a translative covering of almost all of $\Ren$ if and only 
if, $\sum_{F\in\FF} \vol(F)=\infty$.
\end{thm}

Indeed, let $\FF=\{F_1,F_2,\ldots\}$ be a bounded family with infinite total 
volume. Clearly, it is sufficient to cover almost all of the cube 
$C=[-1/2,1/2]^n$. We may assume that $F\subset C$ for all $F\in\FF$. 

We find the translation vectors inductively. Let $x_1=0$. If $x_k$ is 
defined, we denote the uncovered part by
$E_k=C\setminus\left(\mathop\cup\limits_{j=1}^{k} (F_j+x_j)\right)$. We choose 
$x_{k+1}$ in such a way that 
\begin{equation}\label{eq:fkbig}
\frac{\vol\big((F_{k+1}+x_{k+1})\cap E_k\big)}{\vol(F_{k+1})}\geq
\frac{1}{2^n}\vol(E_k).
\end{equation}
It is possible, since
\[
\frac{1}{2^n}\int_{2C}\vol\big((F_{k+1}+x)\cap E_k\big) \di x
=\frac{1}{2^n}\int_{2C}\int_{C}\chi_{F_{k+1}}(y-x)\chi_{E_k}(y)\di y\di x
\]\[
=\frac{1}{2^n}\int_{C}\chi_{E_k}(y)\int_{2C}\chi_{F_{k+1}}(y-x)\di x\di y
=\frac{1}{2^n}\vol(F_{k+1})\vol(E_k).
\]
It is easy to see that \eqref{eq:fkbig} implies that  $\lim_{k\to\infty} 
\vol(E_k)=0$.

We note that the condition that $\FF$ is bounded may be replaced by the 
condition that $\FF$ contains only convex sets, see \cite{Gr85}.

\subsection{A sufficient condition for a family of homothets}
For convex bodies $K$ and $L$, we define $f(K,L)$ as the infimum of 
those $t>0$, such that for any family $\FF$ of homothets of 
$L$ with coefficients $0< \lambda_1, \lambda_2, \dots <1$, the following holds:
\[
\mbox{If } \sum_i\lambda_i^d\geq t \mbox{ then } \FF \mbox{ permits a 
translative covering of } K.
\]
We set 
$
f(n):=\sup\{f(K, K) : K\subset\Ren \mbox{ a convex body}\}.
$

The question of bounding $f(2)$ was originally posed by L. Fejes T\'oth 
\cite{FTL84} (cf. Section~3.2 in \cite{BMP05}). He conjectured that $f(2)\leq 
3$.
Januszewski \cite{Ja03} showed that $f(2)\leq 6.5$. In higher dimensions
Meir and Moser \cite{MeMo68}, and later, A. Bezdek and K. Bezdek \cite{BeBe84} 
considered the cube and proved that $f([0,1]^d)=2^d-1$.
Using a simple argument based on saturated packings by half-sized copies (see 
\eqref{eq:simplepackingbound}), the author \cite{Na10} showed 
\[f(K,L)\leq 
2^n\frac{\vol\left(K+\frac{L\cap(-L)}{2}\right)}{\vol(L\cap(-L))},
\]
from which the bound
\[f(K,K)\leq \left\lbrace
\begin{tabular}{cl}
$3^n,$& if  $K=-K$,\\
$6^n,$& in general.
\end{tabular}\right.
\]
follows.

On the other hand, clearly, $f(K,K)\geq n$ since we may consider $n$ 
homothetic copies of $K$ with homothety ratios slightly below one, and use the 
lower bound on the illumination number of $K$ (see Section~\ref{sec:illum}).

\subsection{A necessary condition for a family of homothets}
A converse to the problem discussed above was 
formulated by V. Soltan \cite{So90} (cf. Section~3.2 in \cite{BMP05}). 
Let 
\[
g(K):=\inf\left\{\sum_i\lambda_i: K\subseteq \bigcup_i \lambda_iK+x_i, 
0<\lambda_i<1\right\},
\] 
and
$g(n):=\inf\{ g(K) : K\subset\Ren \mbox{ a convex body}\}$.
V. Soltan conjectured $g(n)\geq n$.

Since the $n$-dimensional simplex $\Delta$ can be covered by $n+1$ translates 
of $\frac{n}{n+1}\Delta$, we have that $g(n)\leq g(\Delta)\leq n$.
V. Soltan and \'E. V\'as\'arhelyi \cite{SoVa93} showed $g(2)\geq 2$, and 
also proved that the conjecture holds when only $n+1$ homothets are allowed. 

Soltan's conjecture was confirmed in an asymptotic sense in \cite{Na10}:
$\lim\limits_{n\to\infty}\frac{g(n)}{n} = 1$.

%\section*{Acknowledgement}
%The author is grateful for the conversations with

%\bibliographystyle{alpha}
\bibliographystyle{amsalpha}
%\bibliographystyle{amsxport}
%\nocite*
\bibliography{biblio}

\end{document}